\documentclass{siamltex}

\usepackage{amssymb,amsmath,amscd}
\usepackage{algorithmic,algorithm}
\usepackage{dsfont}
\usepackage{graphicx}
\usepackage{color}
\usepackage{subfigure}
\usepackage{framed}
\usepackage{todonotes}
\usepackage{enumerate}
\usepackage{wrapfig}
\usepackage{dpfloat}

\newcommand{\n}[1]{\mathbf{ #1}}

\newcommand{\nv}{\n{n}}

\newcommand{\tv}{{\boldsymbol{\tau}}}

\newcommand{\Grad}{\nabla}

\newcommand{\R}{\mathbb{R}}

\newcommand{\N}{\mathbb{N}}
\newcommand{\vectornorm}[1]{\left \|#1 \right \|}

\newcommand{\scalar}[1]{{\left ( #1 \right )}}
\newcommand{\dual}[1]{{\left \langle #1 \right \rangle}}

\newcommand{\abs}[1]{\left | #1 \right |}

\newcommand{\admshape}{\mathcal{D}}

\newcommand{\dint}{\,d}
\newcommand{\I}{\mathrm{Id}}
\newcommand{\V}{\mathcal{V}}
\newcommand{\Z}{\mathcal{Z}}

\newcommand{\embed}{\hookrightarrow}

\newcommand{\gt}{\tilde{g}}
\newcommand{\Omegab}{\bar{\Omega}}
\newcommand{\Vd}{\n{V}}

\newcommand{\somap}{\mathbf{S}}
\newcommand{\sopot}{\mathbf{S}_p}
\newcommand{\sostokes}{\mathbf{S}_s}

\DeclareMathOperator{\im}{im}

\DeclareMathOperator{\Div}{div}

\newtheorem{remark}[theorem]{Remark}

\title{Approximate Controllability of Linearized Shape-Dependent Operators for Flow Problems}

\author{C. Leith\"auser
	\thanks{Fraunhofer ITWM, Transport Processes, Fraunhofer-Platz 1, 67663 Kaiserslautern, Germany. 
	Email: christian.leithaeuser@itwm.fraunhofer.de}
        \and R. Pinnau
        \thanks{TU Kaiserslautern, Department of Mathematics, Gottlieb-Daimler-Stra{\ss}e, 67663 Kaiserslautern, Germany.}
        \and R. Fe{\ss}ler
        \thanks{Fraunhofer ITWM, Transport Processes, Fraunhofer-Platz 1, 67663 Kaiserslautern, Germany.}}

\begin{document}

\maketitle

\begin{abstract}
 We study the controllability of linearized shape-dependent operators for flow problems. The first operator is a mapping from the shape of the computational domain to the tangential wall velocity of the potential flow problem and the second operator maps to the wall shear stress of the Stokes problem. We derive linearizations of these operators, provide their well-posedness and finally show approximate controllability. The controllability of the linearization shows in what directions the observable can be changed by applying infinitesimal shape deformations.
\end{abstract}



\pagestyle{myheadings}
\thispagestyle{plain}

\section{Introduction}
We study the controllability of linearized shape-dependent operators for flow problems. The first operator $\sopot$ is a mapping from the shape of the computational domain to the tangential wall velocity of the potential flow problem and the second operator $\sostokes$ maps to the wall shear stress of the Stokes problem. On account of the shape dependence, both operators are highly nonlinear, despite of the underlying linear partial differential equations. We investigate linearizations $d \somap$ of these operators, i.e., we study in which directions the observables can by changed by applying infinitesimal shape deformations. Our ultimate goal is to prove approximate controllability for these linearized shape-dependent operators. Approximate controllability means that we can find controls for the operator such that any element from the target space is approximated with arbitrary accuracy. In \cite{leithauser2012characterizing} we have utilized a conformal pull-back to study the operator $\sopot$ directly. However, the approach presented in the following is more general and can be extended to other flow problems, as we are going to show for the Stokes operator $\sostokes$.

Our study of shape-dependent problems is motivated by the optimal shape design of polymer distributors used in the production process for filaments and nonwovens \cite{marburger2007space, leithaeuser2009shape, leithauser2013controllability, leithauser2013numerical}. The goal is to design flow geometries with specific wall shear stress profiles, similar to the problems considered in \cite{quarteroni2003optimal, rozza2005optimization}. Numerically, we can solve the regularized inverse problem of finding a flow geometry which approximately realizes a given wall shear stress, using methods from shape optimization. However, here we address the theoretical question of what wall shear stress profiles are in fact attainable. Being able to establish some sort of controllability property, even though we can only do this for the linearization, suggests that the space of reachable profiles is rather large. For our application this means that we have a good chance to design polymer distributors, whose properties are close to our expectations. This agrees with our numerical results presented in \cite{leithauser2013controllability}, where we solve an optimization problem based on the Stokes operator $\sostokes$.

The controllability of shape-dependent operators is rarely covered in the existing literature. Our approach is inspired by \cite{chenais2003controllability} where the authors study the controllability of a shape identification problem based on the Laplace problem. They show approximate controllability for the linearized operator using an adjoint argument (cf. \cite{osses1998boundary, osses1998controllability}). The operator studied in \cite{chenais2003controllability} is comparable to our operator $\sopot$ because both are based on the Laplace problem. However, the operator in \cite{ chenais2003controllability} maps to the trace evaluated on a fixed interior curve whereas here $\sopot$ maps to the normal derivative evaluated on the variable wall boundary itself, which poses different technical challenges.

A good introduction to the general theory of shape optimization, the concept of shape derivatives and many examples can be found in \cite{pironneau1984optimal} and \cite{sokolowski1992introduction}. The focus in \cite{mohammadi2001applied} is more on the application of industrial airfoil design, but it can also be seen as an excellent access to the general topic. A rigorous treatment of shape derivatives and their existence theory is provided in \cite{simon1980differentiation}. A lot of theory on shape calculus and its application to shape optimization is given in \cite{delfour2010shapes}. Surveys on recent developments are found in \cite{mohammadi2004shape, harbrecht2008analytical}. While we mostly deal with flow problems, there are various other fields of application: For instance, see \cite{allaire2004structural, penzler2010phase, hess2012inexact} for examples on structural optimization, \cite{eppler1990airfoil, mohammadi2001applied} for airfoil design and \cite{sundaramoorthi2009new} for applications in image processing.

We begin in Section \ref{sec:Geometric Setup} by introducing the geometric setup and give proper definitions for the space of admissible shapes and the linearized shape operator. In Section \ref{zua:sec:potential flow} we study the potential flow shape operator $\sopot$, derive its linearization, provide the well-posedness and finally show the approximate controllability. In Section \ref{zua:sec:stokes} we follow the same path for the Stokes operator $\sostokes$. Finally, we finish with a conclusion. In the Appendix, we state some basic facts about shape differentiability and the existence and uniqueness of solutions for partial differential equations (see Appendix \ref{sec:appendix:shape diff} and \ref{sec:appendix:existence}). The main results of this article are stated in Theorem \ref{ds:thm:approximate control pot} and Theorem \ref{ds:thm:approximate control stokes}.

\section{Geometric Setup}
\label{sec:Geometric Setup}
For $k \in \N_0$ let $\Omega_0 \subset \R^2$ be a bounded domain of class $C^{k+1,1}$ (see \cite{wloka1987partial}), where the boundary $\Gamma_0$ decomposes into the in- and outflow parts $\Gamma_0^{in}$ and wall parts $\Gamma_0^w$. Let $\nv$ be the outward pointing unit normal and let $\tv := (-\nv_2, \nv_1)^\intercal$ be the tangential vector. We define
\begin{align}
 \Theta^k = \{\theta \in C^{k,1}(\R^2, \R^2); \vectornorm{\theta}_{C^{k,1}(\R^2, \R^2)} < 0.5 \}
\end{align}
to be a ball around zero, where $C^{k,1}(\R^2, \R^2)$ denotes the space of $k$-times differentiable functions from $\R^2$ to $\R^2$ with Lipschitz-continuous derivatives up to order k (see \cite{wloka1987partial}). Let $\I \in C^{k,1}(\R^2, \R^2)$ denote the identity mapping. For $\theta \in \Theta^k$ we consider the map
\begin{align}
 \I + \theta: \R^2 \rightarrow \R^2,
\end{align}
i.e., $(\I + \theta)(x) = x+\theta(x)$. From \cite{simon1980differentiation} we know that $\vectornorm{\theta}_{C^{k,1}(\R^2, \R^2)} < 0.5$ implies that $\I + \theta$ is a $(k,1)$-diffeomorphism. In order to define the set of admissible shapes let the space of admissible deformation directions be
\begin{align}
\V^k := \{\Vd \in C^{k,1}(\R^2, \R^2); \Vd|_{\Gamma_0^{in}}=0; \Vd|_{\Gamma_0^w} = v_\nv \nv; v_\nv \in C^{k,1}(\R^2) \}.
\end{align}
Note, that since $\Omega_0$ is assumed to be of class $C^{k+1,1}$ we have $\nv \in C^{k,1}(\Gamma_0, \R^2)$. Hence, this definition makes sense. We only consider normal shape deformations, because infinitesimal tangential deformations would shift the boundary along itself and are therefore no real shape deformations. Let the intersection with $\Theta^k$ be denoted by
\begin{align}
 \Theta_\V^k := \Theta^k \cap \V^k.
\end{align}
Then, the space of admissible shapes is given by
\begin{align}
 \admshape^k := \{\Omega_\theta = (\I + \theta)(\Omega_0); \theta \in \Theta_\V^k \}.
\end{align}
Thus, $\admshape^k$ is a set of perturbations of the reference domain $\Omega_0$ which leave $\Gamma_0^{in}$ fixed and which are normal on $\Gamma_0^w$.

\begin{definition}
\label{shape:sec:linearized shape operator}
Let
\begin{align}
 \somap: \admshape^k \rightarrow L^2(\Gamma_0^w).
\end{align}
be a given shape-dependent operator. Then the corresponding linearized shape operator is defined by
\begin{align}
\begin{aligned}
 d\somap: \V^k &\rightarrow L^2(\Gamma_0^w)
 \\
 \Vd &\mapsto \frac{d\somap(\Omega_\theta)}{d\theta}(0) \Vd,
 \end{aligned}
\end{align}
i.e., it is the derivative of $\somap(\Omega_\theta)$ with respect to $\theta$ in direction $\Vd \in \V^k$ evaluated at $\theta=0$.
\end{definition}

Of course the important questions are whether the derivative does exist and how the operator can be evaluated. Both can be answered by the theory of material and shape derivatives which is provided in Appendix \ref{sec:appendix:shape diff}.

Our goal is to show approximate controllability for two different linearized shape operators \cite{chenais2003controllability}:
\begin{definition}[Approximate Controllability]
 Let $F: X \rightarrow Y$ be a linear operator. Then, $F$ is approximately controllable if and only if $\im F$ lies dense in $Y$.
\end{definition}
The definition immediately yields the following lemma which we use to show the property.
\begin{lemma}
\label{lemma:approximately controllable}
 Let $F: X \rightarrow Y$ be a linear operator and let $Y$ be a Hilbert space with scalar product $\scalar{\cdot, \cdot}_Y$. If $y \in Y$ such that
 \begin{align}
  \scalar{F(x), y}_Y = 0
  \qquad \mbox{for all $x \in X$}
 \end{align}
implies $y = 0$, then $F$ is approximately controllable.
\end{lemma}

\section{Potential Flow}
\label{zua:sec:potential flow}

We begin our study with a potential flow shape operator which maps from the shape of the domain to the tangential wall velocity of the potential flow problem. We define the operator and derive its linearization. Then, we use the implicit function theorem to show the existence of the material derivative (see Definition \ref{def:material derivative}) which provides the well-definedness of the linearized shape operator. This also leads to the existence of the shape derivative (see Definition \ref{def:elliptic shape derivative}), which can be computed as the solution of a boundary value problem. We can then write the linearized shape operator in terms of the shape derivative and use an adjoint argument to show that it is approximately controllable.

\subsection{Definition of the Shape Operator and Problem Statement}
Let $\Omega_0 \subset \R^2$ be a bounded domain of class $C^{3,1}$ and let $g \in H^\frac{5}{2}(\Gamma_0)$ be given with $\partial_\tv g|_{\Gamma_0^w} = 0$, where $\partial_\tv$ denotes the derivative in tangential and $\partial_\nv$ the derivative in normal direction. We define the potential flow shape operator $\sopot$ by
\begin{align}
 \begin{aligned}
  \sopot: \admshape^2 &\rightarrow L^2(\Gamma_0^w)
  \\
  \Omega_\theta &\mapsto -(\partial_\nv \Psi(\theta)|_{\Gamma_\theta^w}) \circ (\I + \theta).
 \end{aligned}
\end{align}
Note, that $\partial_\nv \Psi(\theta)|_{\Gamma_\theta^w}$ is a function of $L^2(\Gamma_\theta^w)$ and that we use the the map $\I + \theta$ to pull-back this function to the space $L^2(\Gamma_0^w)$.

For $\theta \in \Theta^2$ the stream function $\Psi(\theta) \in H^2(\Omega_\theta)$ is the solution of
\begin{align}
\label{problem:ds potential}
 \begin{aligned}
  \Delta \Psi(\theta) &= 0 \qquad &&\mbox{in $\Omega_\theta$}
  \\
  \Psi(\theta) &= g \circ (\I+\theta)^{-1} \qquad &&\mbox{on $\Gamma_\theta$}.
 \end{aligned}
\end{align}

\begin{remark}
 The stream function $\Psi(\theta)$ is interpreted as the solution of a flow problem by defining the velocity $\n{u}(\theta)$ through
\begin{align}
 \n{u}(\theta) := 
 \begin{pmatrix} \partial_2 \Psi(\theta) \\ -\partial_1 \Psi(\theta) \end{pmatrix}
 \qquad \mbox{in $\Omega_\theta$}.
 \end{align}
 In that case the normal wall velocity is
 \begin{align}
  \nv \cdot \n{u}(\theta) = \partial_\tv \Psi(\theta) = \partial_\tv (g \circ (\I+\theta)^{-1})
  \qquad \mbox{on $\Gamma_\theta$}
 \end{align}
 and especially $\nv \cdot \n{u}(\theta)|_{\Gamma_\theta^w} = 0$ by definition of $g$. The tangential wall velocity is
 \begin{align}
  \tv \cdot \n{u}(\theta) = -\partial_\nv \Psi(\theta)
  \qquad \mbox{on $\Gamma_\theta$}
 \end{align}
and therefore $\sopot$ maps to the tangential velocity at the wall.

\end{remark}

We are going to show that the linearized shape operator $d\sopot$ is well-defined and given by
\begin{align}
\begin{aligned}
 d\sopot: \V^2 &\rightarrow L^2(\Gamma_0^w)
 \\
 \Vd &\mapsto - \partial_\nv \Psi'(\Vd)|_{\Gamma_0^w} - \kappa \sopot(0) (\nv \cdot \Vd),
 \end{aligned}
\end{align}
 where $\Psi'(\Vd)$ is the solution of
\begin{align}
 \begin{aligned}
  \Delta \Psi'(\Vd) &= 0 \qquad &&\mbox{in $\Omega_0$}
  \\
  \Psi'(\Vd) & = 0 \qquad &&\mbox{on $\Gamma_0^{in}$}
  \\
  \Psi'(\Vd) & = -(\nv \cdot \Vd) \partial_\nv \Psi(0) \qquad &&\mbox{on $\Gamma_0^w$}.
 \end{aligned}
\end{align}
In the rest of this section we establish the existence of $d\sopot$ and prove the following result about the approximate controllability of the linearized shape operator:
\begin{theorem}
\label{ds:thm:approximate control pot}
Assume that $\sopot(0) \neq 0$ a.e. on $\Gamma_0^w$ and suppose that the curvature $\kappa \in C^0(\Gamma_0)$ is positive $\kappa \geq 0$ on $\Gamma_0^w$. Then, $d \sopot$ is approximately controllable.
\end{theorem}

\begin{remark}
 Especially the curvature condition is fulfilled if the wall boundaries are convex. Note, that the statement still holds for curvature $\kappa \geq -\delta$ for a sufficiently small constant $\delta \geq 0$. The constant $\delta$ must be small enough such that the bilinear form  is still $V$-elliptic. Otherwise we can show that the bilinear form is $V$-coercive and prove a weaker result similar to the upcoming Theorem \ref{ds:thm:approximate control stokes}.
\end{remark}

\subsection{Existence of the Material Derivative}
One crucial point is to show the existence of the material derivative for the solution of \eqref{problem:ds potential}, because it gives rise to the well-definedness of the linearized shape operator as well as the existence of the shape derivative. Let us define
\begin{align}
 z(\theta) := \partial_\nv \Psi(\theta)|_{\Gamma_\theta}.
\end{align}
Assume that the material derivative $\dot{z}(\Vd)$ exists for $\Vd \in \V^2$ (see Definition \ref{def:material derivative bd}), then by Definition \ref{shape:sec:linearized shape operator}
\begin{align}
 d\sopot(\Vd) = \frac{d(\sopot(\Omega_\theta))}{d \theta}(0) \Vd
 = -\frac{d(z(\theta) \circ (\I + \theta))}{d \theta}(0) \Vd
 = -\dot{z}(\Vd)|_{\Gamma^w_\theta}.
\end{align}
Therefore, to get a well-defined operator $d\sopot$ we need to show that the material derivative $\dot{z}(\Vd)$ exists. First we show the existence of the material derivative $\dot{\Psi}(\Vd)$ using the implicit function theorem. We need the following regularity result for \eqref{problem:ds potential}:
\begin{lemma}
\label{ds:lemma:regularity pot}
 Let $\Omega_0$ be of class $C^{3,1}$ and assume that $g \in H^\frac{5}{2}(\Gamma_0)$. Then, there exists a unique $\Psi(\theta) \in H^2(\Omega_\theta)$ for every $\theta \in \Theta^2$. Furthermore, $\Psi(0) \in H^3(\Omega_0)$.
\end{lemma}
\begin{proof}
 For $\theta \in \Theta^2$, $\Omega_\theta$ is of class $C^{2,1}$ and $g \circ (\I+\theta)^{-1} \in H^\frac{3}{2}(\Gamma_\theta)$. Therefore, standard existence and regularity theory for elliptic partial differential equations (see \cite{wloka1987partial}) yields $\Psi(\theta) \in H^2(\Omega_\theta)$.
 Furthermore, $\Omega_0$ is of class $C^{3,1}$ and $g \in H^\frac{5}{2}(\Gamma_0)$ which yields $\Psi(0) \in H^3(\Omega_0)$.
\end{proof}

To apply the implicit function theorem we require that the Laplace operator induces an isomorphism:
\begin{lemma}
\label{ds:lemma:iso pot}
 The Laplace operator $\Delta: H^2(\Omega_0) \cap H^1_0(\Omega_0) \rightarrow L^2(\Omega_0)$ is an isomorphism between the given spaces.
\end{lemma}
\begin{proof}
 The operator is clearly linear. Let $f \in H^2(\Omega_0) \cap H^1_0(\Omega_0)$ then $\Delta f \in L^2(\Omega_0)$. On the other hand let $h \in L^2(\Omega_0)$, then there exists a unique solution $f \in H^2(\Omega_0) \cap H^1_0(\Omega_0)$ of $\Delta f = h$ (see \cite{wloka1987partial}).
\end{proof}

Now, we can show the existence of the material derivative $\dot{\Psi}(\Vd)$ of $\Psi(\theta)$. The proof relies on the inverse function theorem and the idea can be found in \cite{simon1980differentiation, sokolowski1992introduction}.
\begin{lemma}
\label{ds:lemma:existence material pot}
 Suppose that for the solution of Problem \eqref{problem:ds potential}, $\Psi(\theta) \in H^2(\Omega_\theta)$ holds for $\theta \in \Theta^2$. Then, the material derivative $\dot{\Psi}(\Vd) \in H^2(\Omega_0)$ exists for all directions $\Vd \in C^{2,1}(\R^2, \R^2)$.
\end{lemma}
\begin{proof}
 Let $\gt \in H^3(\Omega_0)$ be an extension with $\gt|_{\Gamma_0} = g$. Let us define the function
 \begin{align}
  \begin{aligned}
   F: \Theta^2 \times H^2(\Omega_0) \cap H^1_0(\Omega_0) &\rightarrow L^2(\Omega_0)
   \\
   (\theta, u) &\mapsto \Delta_\theta u + \Delta_\theta \gt.
  \end{aligned}
 \end{align}
 See Lemma \ref{ds:lemma:gradf_TV} for the definition of the pulled-back Laplacian $\Delta_\theta$. 
 
 Let $\theta \in \Theta^2$. Then,
 \begin{align}
  \Delta \Psi(\theta) = 0 \qquad \mbox{in $\Omega_\theta = (\I + \theta)(\Omega_0)$}
 \end{align}
 and thus
 \begin{align}
  (\Delta \Psi(\theta)) \circ (\I + \theta) = 0 \qquad \mbox{in $\Omega_0$}.
 \end{align}
 Using Lemma \ref{ds:lemma:gradf_TV} this implies
 \begin{align}
  \Delta_\theta (\Psi(\theta) \circ (\I + \theta)) = 0  \qquad \mbox{in $\Omega_0$},
 \end{align}
 where $\Psi(\theta) \circ (\I + \theta) - \gt \in H^2(\Omega_0) \cap H^1_0(\Omega_0)$ and thus
 \begin{align}
 \label{ds:eq:material derivative pot eq1}
  F(\theta, \Psi(\theta) \circ (\I + \theta) - \gt) = 0.
 \end{align}
 
Let $u_0 := \Psi(0) -\gt \in H^2(\Omega_0) \cap H^1_0(\Omega_0)$. Then, $F(0, u_0) = 0$ and
 \begin{align}
  D_2 F(0, u_0) = \Delta: H^2(\Omega_0) \cap H^1_0(\Omega_0) \rightarrow L^2(\Omega)
 \end{align}
 is an isomorphism by Lemma \ref{ds:lemma:iso pot}. Furthermore, from \cite[(1.3)]{simon1980differentiation} we know that the operator $\Delta_\theta$ is differentiable with respect to $\theta$ and thus that $F$ is differentiable, i.e.,
 \begin{align}
  F \in C^1(\Theta^2 \times H^2(\Omega_0) \cap H^1_0(\Omega_0), L^2(\Omega_0)).
 \end{align}
 
 Then, because of the Implicit Function Theorem \ref{thm:implicit function} there exists a unique $\mathcal{G} \in C^1(\Theta^2, H^2(\Omega_0) \cap H^1_0(\Omega_0))$ and \eqref{ds:eq:material derivative pot eq1} implies
 \begin{align}
  \mathcal{G}(\theta) = \Psi(\theta) \circ (\I + \theta) - \gt
 \end{align}
 for $\theta \in \Theta^2$. Then,
 \begin{align}
  \Psi(\theta) \circ (\I + \theta) = \mathcal{G}(\theta) + \gt
 \end{align}
 is differentiable with respect to $\theta$ at $\theta=0$ and the derivative lies in $H^2(\Omega_0)$. Thus the material derivative $\dot{\Psi}(\Vd) \in H^2(\Omega_0)$ exists for $\Vd \in C^{2,1}(\R^2, \R^2)$.
\end{proof}

This yields the well-definedness of the linearized shape operator:
\begin{lemma}
\label{lemma:existence mat bd pot}
 The material derivative $\dot{z}(\Vd) \in H^{\frac{1}{2}}(\Gamma_0)$ exists for $\Vd \in C^{2,1}(\R^2, \R^2)$. Thus the operator $d \sopot$ is well-defined.
\end{lemma}
\begin{proof}
 Let $\Vd \in C^{2,1}(\R^2, \R^2)$. We know from Lemma \ref{ds:lemma:existence material pot} that $\dot{\Psi}(\Vd) \in H^2(\Omega_0)$ which implies the existence of $\dot{z}(\Vd) \in H^{\frac{1}{2}}(\Gamma_0)$ (see \cite{leithauser2013controllability}). For $\Vd \in \V^2$ we have $d \sopot(\Vd) = -\dot{z}(\Vd) \in L^2(\Gamma_0^w)$ and the operator is well-defined.
\end{proof}

\subsection{Existence of the Shape Derivative}
Computing the operator $d \sopot$ in an explicit way can be done via the shape derivative. The existence of the shape derivative can be derived from the existence of the material derivative and the following Lemma gives an explicit form for $\Psi'(\Vd)$.
\begin{lemma}
\label{ds:lemma:pot shape div psi}
 For $\theta \in \Theta^2$ let $\Psi(\theta) \in H^2(\Omega_\theta)$ be the solution of \eqref{problem:ds potential}, then for $\Vd \in C^{2,1}(\R^2, \R^2)$ the shape derivative $\Psi'(\Vd) \in H^2(\Omega_0)$ exists and can be computed as the solution of
 \begin{align}
\label{eq:Potential Flow Shape derivative 1}
 \begin{aligned}
  \Delta \Psi'(\Vd) &= 0 \qquad &&\mbox{in $\Omega_0$}
  \\
  \Psi'(\Vd) & = 0 \qquad &&\mbox{on $\Gamma_0^{in}$}
  \\
  \Psi'(\Vd) & = -(\nv \cdot \Vd) \partial_\nv \Psi(0) \qquad &&\mbox{on $\Gamma_0^w$}.
 \end{aligned}
\end{align}
\end{lemma}
\begin{proof}
By Lemma \ref{ds:lemma:regularity pot}, $\Psi(0) \in H^3(\Omega_0)$ and by Lemma \ref{ds:lemma:existence material pot} the material derivative $\dot{\Psi}(\Vd) \in H^2(\Omega_0)$ exists for $\Vd \in C^{2,1}(\R^2, \R^2)$. Then, by Definition \ref{def:elliptic shape derivative} the shape derivative $\Psi'(\Vd) \in H^2(\Omega_0)$ exists. Furthermore, \cite[Proposition 3.1]{sokolowski1992introduction} yields that $\Psi'(\Vd)$ solves \eqref{eq:Potential Flow Shape derivative 1}.
\end{proof}

\begin{lemma}
\label{zua:lemma:normal shape derivative 1}
 For $z(\theta) = \partial_\nv \Psi(\theta)|_{\Gamma_\theta} \in H^\frac{1}{2}(\Gamma_\theta)$, $\theta \in \Theta^k$ the shape derivative $z'(\Vd) \in H^{\frac{1}{2}}(\Gamma_0)$ exists for $\Vd \in C^{2,1}(\R^2, \R^2)$ and it is given on the wall boundaries by
 \begin{align}
  z'(\Vd) = \partial_\nv \Psi'(\Vd)|_{\Gamma_0^w} - \kappa z(0) (\nv \cdot \Vd)
  \qquad \mbox{on $\Gamma_0^w$}.
 \end{align}
\end{lemma}
\begin{proof}
Let $\Vd \in C^{2,1}(\R^2, \R^2)$. We have shown in Lemma \ref{lemma:existence mat bd pot} that the material derivative $\dot{z}(\Vd) \in H^\frac{1}{2}(\Gamma_0)$ exists. Furthermore, we know that $\Psi(0) \in H^3(\Omega_0)$ and thus $z(0) \in H^\frac{3}{2}(\Gamma_0)$ by the Trace Theorem (see \cite{wloka1987partial}). Then, Definition \ref{def:elliptic bd shape derivative} yields the existence of $z'(\Vd) \in H^\frac{1}{2}(\Gamma_0)$.

Next, we show that the shape derivative has the given form on the wall boundaries. Therefore, let $\Vd \in C^{2,1}(\R^2, \R^2)$ be given. Let $\phi \in C^\infty(\R^2)$ with $\partial_\nv \phi = 0$ on $\Gamma_0$ and $\phi = 0$ on $\Gamma_0^{in}$ be a smooth test function. For $\theta \in \Theta^2$, integration by parts yields
 \begin{align}
  \begin{aligned}
   0 &= \int_{\Omega_\theta} \Delta \Psi(\theta) \phi \dint x
   \\
   &= -\int_{\Omega_\theta} \Grad \Psi(\theta) \cdot \Grad \phi \dint x 
   + \int_{\Gamma_\theta} z(\theta) \phi \dint s.
  \end{aligned}
 \end{align}
 Using Lemma \ref{zua:lemma:diff domain int} and Lemma \ref{zua:lemma:diff boundary int} to differentiate with respect to $\theta$ in direction $\Vd \in C^{2,1}(\R^2, \R^2)$ yields
 \begin{align}
 \label{zua:eq:normal shape derivative eq1}
  \begin{aligned}
   -&\int_{\Omega_0} \Grad \Psi'(\Vd) \cdot \Grad \phi \dint x 
   - \int_{\Gamma_0} \Grad \Psi(0) \cdot \Grad \phi \,(\nv \cdot \Vd) \dint s
   \\
   + &\int_{\Gamma_0} (z'(\Vd) \phi + \kappa z(0) \phi \,(\nv \cdot \Vd)) \dint s = 0.
  \end{aligned}
 \end{align}
 By Lemma \ref{ds:lemma:pot shape div psi} we have $\Delta \Psi'(\Vd) = 0$ in $\Omega_0$ and integration by parts yields
 \begin{align}
  \label{zua:eq:normal shape derivative eq2}
  \int_{\Omega_0} \Grad \Psi'(\Vd) \cdot \Grad \phi \dint x = \int_{\Gamma_0^w} \partial_\nv \Psi'(\Vd) \phi \dint s.
 \end{align}
 On the other hand, it holds
 \begin{align}
  \label{zua:eq:normal shape derivative eq3}
 \begin{aligned}
  &\int_{\Gamma_0} \Grad \Psi(0) \cdot \Grad \phi \,(\nv \cdot \Vd) \dint s
  \\
  =& \int_{\Gamma_0^w} \partial_\nv \Psi(0) \, \partial_\nv \phi \,(\nv \cdot \Vd) \dint s
  + \int_{\Gamma_0^w} \partial_\tv \Psi(0) \, \partial_\tv \phi \,(\nv \cdot \Vd) \dint s
  \\
  =&\, 0.
  \end{aligned}
 \end{align}
 where we have used $\nv \cdot \Vd = 0$ on $\Gamma_0^{in}$ and $\partial_\nv \phi = 0$, $\partial_\tv \Psi(0) = \partial_\tv g = 0$ on $\Gamma_0^w$. Then, plugging \eqref{zua:eq:normal shape derivative eq2} and \eqref{zua:eq:normal shape derivative eq3} into \eqref{zua:eq:normal shape derivative eq1} and using that $\phi$ vanishes on $\Gamma_0^{in}$ yields
 \begin{align}
  \int_{\Gamma_0^w} (-\partial_\nv \Psi'(\Vd) + z'(\Vd) + \kappa z(0) (\nv \cdot \Vd)) \phi \dint s = 0.
 \end{align}
 And since $\phi \in C^\infty(\R^2)$ is arbitrary on $\Gamma_0^w$ and dense in $L^2(\Gamma_0^w)$ we conclude
 \begin{align}
  z'(\Vd) = \partial_\nv \Psi'(\Vd)|_{\Gamma_0^w} - \kappa z(0) (\nv \cdot \Vd).
 \end{align}
\end{proof}

\begin{lemma}
\label{lemma:lin operator pot}
The linearized shape operator $d\sopot$ is well-defined and given by
\begin{align}
\begin{aligned}
 d\sopot: \V^2 &\rightarrow L^2(\Gamma_0^w)
 \\
 \Vd &\mapsto - \partial_\nv \Psi'(\Vd)|_{\Gamma_0^w} + \kappa z(0) (\nv \cdot \Vd),
 \end{aligned}
\end{align}
 where $\Psi'(\Vd)$ is the solution of
\begin{align}
 \begin{aligned}
  \Delta \Psi'(\Vd) &= 0 \qquad &&\mbox{in $\Omega_0$}
  \\
  \Psi'(\Vd) & = 0 \qquad &&\mbox{on $\Gamma_0^{in}$}
  \\
  \Psi'(\Vd) & = -(\nv \cdot \Vd) \partial_\nv \Psi(0) \qquad &&\mbox{on $\Gamma_0^w$}.
 \end{aligned}
\end{align}
\end{lemma}

\begin{proof}
 We have shown in Lemma \ref{lemma:existence mat bd pot} that the material derivative of $z(\theta)$ exists and thus that the operator is well-defined. Let $\Vd \in \V^2$. Remember that by definition $\Vd$ is normal on $\Gamma_0$. We conclude using Definition \ref{def:elliptic bd shape derivative} and Lemma \ref{zua:lemma:normal shape derivative 1}
 \begin{align}
\label{eq: dspot equal to shape div}
 d\sopot(\Vd) = -\dot{z}(\Vd) = -z'(\Vd) + \partial_\tv z(0) (\tv \cdot \Vd)
 = - \partial_\nv \Psi'(\Vd)|_{\Gamma_0^w} + \kappa z(0) (\nv \cdot \Vd).
\end{align}
\end{proof}

\subsection{Approximate Controllability}
We have derived the linearized potential flow shape operator and use it to show our approximate controllability result. To do this we need the following uniqueness lemma:
\begin{lemma}
\label{zua:lemma:Uniqueness Potential flow}
 Assume that the curvature $\kappa \in C^0(\Gamma_0)$ is nonnegative, i.e.,  $\kappa \geq 0$ on $\Gamma_0^w$. If $\phi \in H^2(\Omega_0)$ solves
 \begin{align}
 \label{zua:eq:Uniqueness Potential flow}
 \begin{aligned}
  \Delta \phi &= 0
  \qquad &&\mbox{in $\Omega_0$}
  \\
  \phi &= 0
  \qquad && \mbox{on $\Gamma_0^{in}$}
  \\
  \partial_\nv \phi + \kappa \phi &= 0
  \qquad && \mbox{on $\Gamma_0^w$}
 \end{aligned}
 \end{align}
 then $\phi = 0$.
\end{lemma}
\begin{proof}
 Define the space $V := \{ y \in H^1(\Omega_0); y|_{\Gamma_0^{in}} = 0 \}$. Let $\phi \in H^2(\Omega_0)$ solve \eqref{zua:eq:Uniqueness Potential flow}. Testing the equation with $\phi$ yields after integration by parts
 \begin{align}
  \begin{aligned}
   0 &= -\int_{\Omega_0} \Delta \phi \, \phi \dint x
   \\
   &= \int_{\Omega_0} \| \Grad \phi\|^2 \dint x + \int_{\Gamma_0^w} \kappa \phi^2 \dint s.
  \end{aligned}
 \end{align}
Due to $\kappa \ge 0$ this implies $\phi \equiv const$ a.e. in $\Omega_0$ and the Dirichlet condition yields $\phi \equiv 0 $ a.e. in $\Omega_0$.

\end{proof}

Finally, we have everything at hand to show the approximate controllability for $d \sopot$ using an adjoint argument.

\begin{proof}[of Theorem \ref{ds:thm:approximate control pot}]
Define
\begin{align}
 H^\frac{3}{2}_{i=0}(\Gamma_0) = \{\mu \in H^\frac{3}{2}(\Gamma_0); \mu = 0 \mbox{ on $\Gamma_0^{in}$} \}
\end{align}
 and for $\mu \in H^\frac{3}{2}_{i=0}(\Gamma_0)$ let $\phi(\mu) \in H^2(\Omega_0)$ be the unique solution of the adjoint problem
 \begin{align}
 \begin{aligned}
  \Delta \phi(\mu) &= 0
  \qquad &&\mbox{in $\Omega_0$}
  \\
  \phi(\mu) &= \mu
  \qquad && \mbox{on $\Gamma_0$},
 \end{aligned}
 \end{align}
 which has a unique and regular solution (c.f.~\cite{wloka1987partial}).
 For $(\Vd, \mu) \in \V^2 \times H^\frac{3}{2}_{i=0}(\Gamma_0)$ integration by parts yields
 \begin{align}
  \begin{aligned}
   0 &= \int_{\Omega_0} \Delta \Psi'(\Vd)  \phi(\mu) \dint x
   \\
   &= \int_{\Omega_0} \Psi'(\Vd)  \Delta \phi(\mu) \dint x
   + \int_{\Gamma_0^w} \partial_\nv \Psi'(\Vd) \phi(\mu) \dint s
   - \int_{\Gamma_0^w} \Psi'(\Vd) \partial_\nv \phi(\mu) \dint s
  \end{aligned}
 \end{align}
 and therefore
 \begin{align}
  \int_{\Gamma_0^w} \partial_\nv \Psi'(\Vd) \mu \dint s
   = -\int_{\Gamma_0^w} (\nv \cdot \Vd) \partial_\nv \Psi(0) \partial_\nv \phi(\mu) \dint s.
 \end{align}
 
 Now, assume that $\mu \in \im(d\sopot)^\perp \cap H^\frac{3}{2}_{i=0}(\Gamma_0)$, i.e.,
 \begin{align}
  \int_{\Gamma_0^w} d\sopot(\Vd) \mu \dint s = 0
  \qquad \mbox{for all $\Vd \in \V^2$}
 \end{align}
 holds. Then, we conclude
 \begin{align}
  \begin{aligned}
   0 &= \int_{\Gamma_0^w} d\sopot(\Vd) \mu \dint s
   \\
   &= -\int_{\Gamma_0^w} \partial_\nv \Psi'(\Vd) \mu \dint s
   + \int_{\Gamma_0^w} \kappa \partial_\nv \Psi(0) (\nv \cdot \Vd) \mu \dint s
   \\
   &= \int_{\Gamma_0^w} \big(\partial_\nv \phi(\mu) + \kappa \phi(\mu)\big) (\nv \cdot \Vd) \partial_\nv \Psi(0) \dint s.
  \end{aligned}
 \end{align}
 Now, by assumption $\partial_\nv \Psi(0) = -\sopot(0) \neq 0$ a.e. on $\Gamma_0^w$, therefore,
\begin{align}
 \{(\nv \cdot \Vd) (\partial_\nv \Psi(0))|_{\Gamma_0^w}; \Vd \in \V^2 \}
\end{align}
 is dense in $L^2(\Gamma_0^w)$ and we conclude
 \begin{align}
  \partial_\nv \phi(\mu) + \kappa \phi(\mu) = 0
  \qquad \mbox{on $\Gamma_0^w$}.
 \end{align}
 This leads to a problem independent of $\mu$:
 \begin{align}
 \begin{aligned}
  \Delta \phi(\mu) &= 0
  \qquad &&\mbox{in $\Omega_0$}
  \\
  \phi(\mu) &= 0
  \qquad && \mbox{on $\Gamma_0^{in}$}
  \\
  \partial_\nv \phi(\mu) + \kappa \phi(\mu) &= 0
  \qquad && \mbox{on $\Gamma_0^w$}.
 \end{aligned}
 \end{align}
 Lemma \ref{zua:lemma:Uniqueness Potential flow} yields that $\phi(\mu) = 0$ is the only solution which implies $\mu = \phi(\mu)|_{\Gamma_0} = 0$. Then, Lemma \ref{lemma:approximately controllable} yields that $d\sopot$ is approximately controllable.
\end{proof}

Thus, we have shown that the linearized shape operator of this potential flow problem is approximately controllable.

\section{Stokes Flow}
\label{zua:sec:stokes}

We want to continue with an operator based on the Stokes equation, which maps to the wall shear stress. This operator is motivated by our application of designing optimal distributor geometries for polymer spin packs. We want to generate a better understanding on the inverse problem of finding a flow geometry with a certain wall shear stress profile. Especially, we want to explore whether the space of reachable profiles is rather large or small. We show that the operator $d \sostokes$ is controllable in the sense of Theorem \ref{ds:thm:approximate control stokes}. This backs our expectations on the numerics and we can hope to design distributor geometries with a wall shear stress close to the desired target stress.

\subsection{Definition of the Shape Operator and Problem Statement}
Let $\Omega_0 \subset \R^2$ be a bounded domain of class $C^{6,1}$ and let $g \in H^{5+\frac{1}{2}}(\Gamma_0)$ be given with $\partial_\tv g|_{\Gamma_0^w} = 0$. See Remark \ref{remark:high regularity} for a justification of the high regularity requirement. We define the Stokes flow shape operator $\sostokes$ by
\begin{align}
 \begin{aligned}
  \sostokes: \admshape^4 &\rightarrow L^2(\Gamma_0^w)
  \\
  \Omega_\theta &\mapsto (\omega(\theta)|_{\Gamma_\theta^w}) \circ (\I + \theta).
 \end{aligned}
\end{align}
For $\Theta^4$ the stream function $\Psi(\theta)$ and vorticity $\omega(\theta)$ are the solutions of
\begin{align}
\label{problem:ds biharmonic stokes}
\begin{aligned}
 \Delta \Psi(\theta) &= -\omega(\theta) \qquad &&\mbox{in $\Omega_\theta$}
 \\
 \Delta \omega(\theta) &= 0 \qquad &&\mbox{in $\Omega_\theta$}
 \\
 \Psi(\theta) &= g \circ (\I + \theta)^{-1} \qquad &&\mbox{on $\Gamma_\theta$}
 \\
 \partial_\nv \Psi(\theta) &= 0 \qquad &&\mbox{on $\Gamma_\theta$}.
 \end{aligned}
\end{align}

\begin{remark}
The flow velocity is given by
\begin{align}
 \n{u}(\theta) = \begin{pmatrix} \partial_2 \Psi(\theta) \\ -\partial_1 \Psi(\theta) \end{pmatrix}
\end{align}
and $\n{u}(\theta)$ solves Stokes equation (see \cite{anderson1995computational})
\begin{align}
\begin{aligned}
 -\Delta \n{u}(\theta) + \Grad p &= 0
 \qquad &&\mbox{in $\Omega_\theta$}
 \\
 \Div \n{u}(\theta) &= 0
 \qquad &&\mbox{in $\Omega_\theta$}
\end{aligned}
\end{align}
with boundary conditions
\begin{align}
\begin{aligned}
 \tv \cdot \n{u}(\theta) &= -\partial_\nv \Psi(\theta) = 0
  \qquad &&\mbox{on $\Gamma_\theta$}
  \\
  \nv \cdot \n{u}(\theta) &= \partial_\tv \Psi(\theta) = \partial_\tv (g \circ (\I+\theta)^{-1})
  \qquad &&\mbox{on $\Gamma_\theta$}
\end{aligned}
\end{align}
 and especially $\nv \cdot \n{u}(\theta)|_{\Gamma_\theta^w} = 0$ by definition of $g$. Furthermore,  $\sostokes$ maps to the wall shear stress $\sigma(\theta) = \omega(\theta)|_{\Gamma_\theta^w}$.
\end{remark}

We show that the linearized shape operator $d\sostokes$ is well-defined and given by
\begin{align}
\begin{aligned}
 d\sostokes: \V^4 &\rightarrow L^2(\Gamma_0^w)
 \\
 \Vd &\mapsto \omega'(\Vd)|_{\Gamma_0^w} + \partial_\nv \omega(0) (\nv \cdot \Vd).
 \end{aligned}
\end{align}
 where $\Psi'(\Vd)$ and $\omega'(\Vd)$ are the solution of
 \begin{align}
 \label{zua:eq:BiharmonicShapeDerivative}
 \begin{aligned}
  \Delta \Psi'(\Vd) &= -\omega'(\Vd)
  \qquad &&\mbox{in $\Omega_0$}
  \\
  \Delta \omega'(\Vd) &= 0
  \qquad &&\mbox{in $\Omega_0$}
  \\
  \Psi'(\Vd) &= 0
  \qquad &&\mbox{on $\Gamma_0$}
  \\
  \partial_\nv \Psi'(\Vd) &= 0
  \qquad &&\mbox{on $\Gamma_0^{in}$}
  \\
  \partial_\nv \Psi'(\Vd) &= (\nv \cdot \Vd) \omega(0)
  \qquad &&\mbox{on $\Gamma_0^w$}.
 \end{aligned}
 \end{align}
In the rest of this section we establish the existence of $d\sostokes$ and prove the following result about the approximate controllability of the linearized shape operator:
\begin{theorem}
\label{ds:thm:approximate control stokes}
 Let $\Omega_0$ be bounded and of class $C^{6,1}$ and assume that $\sostokes(0) \neq 0$ on $\Gamma_0^w$. Then, the operator $d\sostokes: \V^4 \rightarrow L^2(\Gamma_0^w)/\mathcal{Z}_{\partial_\nv}$ is approximately controllable. Here $\mathcal{Z}_{\partial_\nv} = \{ \partial_\nv \phi|_{\Gamma_0^w} \in L^2(\Gamma_0^w); \phi \in H^4(\Omega_0) \mbox{ solution of \eqref{zua:eq:BiharmonicUniqueness}} \}$ is a finite dimensional subspace of $L^2(\Gamma_0^w)$.
\end{theorem}

\begin{remark}
\label{remark:high regularity}
 The assumptions of this section include a very high regularity requirement of $C^{6,1}$ for the reference domain $\Omega_0$. For the well-definedness of the operator $\sopot$ itself, $C^{4,1}$ would suffice, because this would provide the existence of the trace of $\omega(\theta)$. It is also true that in many parts of this section the regularity assumptions can be relaxed by applying weak arguments. However, a key part for the final proof is the $V$-coercivity of the bilinear form \eqref{zua:eq:Stokes Uniqueness weak}, which due to \cite{wloka1987partial} does require $c_{11} \in C^1(\Omegab_0)$ for the coefficient of the boundary form. And by definition of that coefficient this requires $C^{6,1}$ for $\Omega_0$ (cf. Lemma \ref{ds:lemma:regularity c11}).
\end{remark}

\subsection{Existence of the Material Derivative}
To prove the well-posedness of the linearized shape operator $d \sostokes$ let us define
\begin{align}
 z(\theta) := \omega(\theta)|_{\Gamma_\theta} = \Delta \Psi(\theta)|_{\Gamma_\theta}.
\end{align}
Again, our first task is to show the existence of the material derivative $\dot{\Psi}(\Vd)$ of the stream function as the solution of the biharmonic problem
\begin{align}
\label{ds:eq:biharmonic stokes 1}
\begin{aligned}
 \Delta \Delta \Psi(\theta) &= 0 \qquad &&\mbox{in $\Omega_\theta$}
 \\
 \Psi(\theta) &= g \circ (\I + \theta)^{-1} \qquad &&\mbox{on $\Gamma_\theta$}
 \\
 \partial_\nv \Psi(\theta) &= 0 \qquad &&\mbox{on $\Gamma_\theta$}.
 \end{aligned}
\end{align}
We start by stating the standard regularity result:
\begin{lemma}
\label{ds:lemma:regularity BHSTO}
 For $\theta \in \Theta^4$ let $\Psi(\theta)$ be the solution of Problem \eqref{ds:eq:biharmonic stokes 1}, then $\Psi(\theta) \in H^4(\Omega_\theta)$. Furthermore, $\Psi(0) \in H^6(\Omega_0)$.
\end{lemma}
\begin{proof}
 Let $\theta \in \Theta^4$, then $\Omega_\theta \in C^{4,1}$ and $g \circ (\I+\theta)^{-1} \in H^{5+\frac{1}{2}}(\Gamma_\theta)$. The standard existence and regularity theory (see \cite{wloka1987partial}) implies $\Psi(\theta) \in H^4(\Omega_\theta)$. Furthermore, since $\Omega_0$ is of class $C^{6,1}$ and $g \in H^{5+\frac{1}{2}}(\Gamma_0)$ we have $\Psi(0) \in H^6(\Omega_0)$.
\end{proof}

In the same way as for the Laplace operator (cf. Lemma \ref{ds:lemma:iso pot}) the elliptic existence and regularity theory yields:
\begin{lemma}
\label{ds:lemma:iso biha}
 The biharmonic operator $\Delta \Delta: H^4(\Omega_0) \cap H^2_0(\Omega_0) \rightarrow L^2(\Omega_0)$ is an isomorphism between the given spaces.
\end{lemma}

Again, we use the implicit function theorem to show the existence of the material derivative (cf. \cite{simon1980differentiation, sokolowski1992introduction}).

\begin{lemma}
\label{ds:lemma:material derivative stokes}
 Suppose that for the solution of \eqref{problem:ds biharmonic stokes} fulfills $\Psi(\theta) \in H^4(\Omega_\theta)$ for $\theta \in \Theta^4$. Then, the material derivative $\dot{\Psi}(\Vd) \in H^4(\Omega_0)$ exists for all directions $\Vd \in C^{4,1}(\R^2, \R^2)$.
\end{lemma}

\begin{proof}
 Let $\gt \in H^6(\Omega_0)$ be an extension with $\gt|_{\Gamma_0} = g$ and $\partial_\nv \gt|_{\Gamma_0} = 0$. Define the function
 \begin{align}
 \begin{aligned}
  F: \Theta^4 \times H^4(\Omega_0) \cap H_0^2(\Omega_0) &\rightarrow L^2(\Omega_0)
  \\
  (\theta, u) &\mapsto \Delta_\theta \Delta_\theta u + \Delta_\theta \Delta_\theta \gt.
 \end{aligned}
 \end{align}
 
 Let $\theta \in \Theta^4$. Then, it holds
 \begin{align}
  \Delta \Delta \Psi(\theta) = 0 \qquad \mbox{in $\Omega_\theta = (\I + \theta)(\Omega_0)$}
 \end{align}
 and thus
 \begin{align}
  (\Delta \Delta \Psi(\theta)) \circ (\I + \theta) = 0 \qquad \mbox{in $\Omega_0$}.
 \end{align}
 Using Lemma \ref{ds:lemma:gradf_TV} this implies
 \begin{align}
  \Delta_\theta \Delta_\theta (\Psi(\theta) \circ (\I + \theta)) = 0  \qquad \mbox{in $\Omega_0$},
 \end{align}
 where $\Psi(\theta) \circ (\I + \theta) - \gt \in H^4(\Omega_0) \cap H^2_0(\Omega_0)$ and thus
 \begin{align}
 \label{ds:eq:material derivative stokes eq1}
  F(\theta, \Psi(\theta) \circ (\I + \theta) - \gt) = 0.
 \end{align}
 
 Let $u_0 := \Psi(0) -\gt \in H^4(\Omega_0) \cap H^2_0(\Omega_0)$. Then, $F(0, u_0) = 0$ and
 \begin{align}
  D_2 F(0, u_0) = \Delta \Delta: H^4(\Omega_0) \cap H^2_0(\Omega_0) \rightarrow L^2(\Omega)
 \end{align}
 is an isomorphism by Lemma \ref{ds:lemma:iso biha}. Furthermore, from \cite[(1.3)]{simon1980differentiation} we conclude that the operator $F$ is differentiable, i.e.,
 \begin{align}
  F \in C^1(\Theta^4 \times H^4(\Omega_0) \cap H^2_0(\Omega_0), L^2(\Omega_0)).
 \end{align}
 
 Then, because of the Implicit Function Theorem \ref{thm:implicit function} there exists a unique $\mathcal{G} \in C^1(\Theta^4, H^4(\Omega_0) \cap H^2_0(\Omega_0))$ and Equation \eqref{ds:eq:material derivative stokes eq1} implies
 \begin{align}
  \mathcal{G}(\theta) = \Psi(\theta) \circ (\I + \theta) - \gt
 \end{align}
 for $\theta \in \Theta^4$. Then,
 \begin{align}
  \Psi(\theta) \circ (\I + \theta) = \mathcal{G}(\theta) + \gt
 \end{align}
 is differentiable with respect to $\theta$ at $\theta=0$ where the derivative lies in $H^4(\Omega_0)$. Thus, the material derivative $\dot{\Psi}(\Vd) \in H^4(\Omega_0)$ exists for $\Vd \in C^{4,1}(\R^2, \R^2)$.
\end{proof}

Now where we have established the existence of $\dot{\Psi}(\Vd)$, the existence of $\dot{\omega}(\Vd)$ and $\dot{z}(\Vd)$, with $z(\theta) = \omega(\theta)|_{\Gamma_\theta}$ follow directly:
\begin{lemma}
\label{lemma:ds stokes bd mat existence}
 The material derivative $\dot{\omega}(\Vd) \in H^2(\Omega_0)$ exists for $\Vd \in C^{4,1}(\R^2, \R^2)$. Let $z(\theta) = \omega(\theta)|_{\Gamma_\theta}$ for $\theta \in \Theta^4$. Then, the material derivative $\dot{z}(\Vd) \in H^\frac{3}{2}(\Gamma_0)$ exists for $\Vd \in C^{4,1}(\R^2, \R^2)$. Thus the operator $d \sostokes$ is well-defined.
\end{lemma}
\begin{proof}
 Let $\Vd \in C^{4,1}(\R^2, \R^2)$. By Lemma \ref{ds:lemma:material derivative stokes}, $\dot{\Psi}(\Vd) \in H^4(\Omega_0)$ which implies $\dot{\omega}(\Vd) \in H^2(\Omega_0)$ and thus $\dot{z}(\Vd) \in H^\frac{3}{2}(\Gamma_0)$ by Lemma \ref{shape:lemma:boundary material derivative}.
\end{proof}

\subsection{Existence of the Shape Derivative}
Since we have shown the existence of the material derivatives we get the following result for the shape derivatives.
\begin{lemma}
For $\Vd \in C^{4,1}(\R^2, \R^2)$ the shape derivatives $\Psi'(\Vd) \in H^4(\Omega_0)$ and $\omega'(\Vd) \in H^2(\Omega_0)$ exist. Furthermore, for $\Vd \in \V^4$ it is given as the solution of
 \begin{align}
 \begin{aligned}
  \Delta \Psi'(\Vd) &= -\omega'(\Vd)
  \qquad &&\mbox{in $\Omega_0$}
  \\
  \Delta \omega'(\Vd) &= 0
  \qquad &&\mbox{in $\Omega_0$}
  \\
  \Psi'(\Vd) &= 0
  \qquad &&\mbox{on $\Gamma_0$}
  \\
  \partial_\nv \Psi'(\Vd) &= 0
  \qquad &&\mbox{on $\Gamma_0^{in}$}
  \\
  \partial_\nv \Psi'(\Vd) &= (\nv \cdot \Vd) \omega(0)
  \qquad &&\mbox{on $\Gamma_0^w$}.
 \end{aligned}
 \end{align}
\end{lemma}
\begin{proof}
Let $\Vd \in C^{4,1}(\R^2, \R^2)$. We have shown that $(\dot{\Psi}(\Vd), \dot{\omega}(\Vd)) \in H^4(\Omega_0) \times H^2(\Omega_0)$ exists and that $(\Psi(0), \omega(0)) \in H^6(\Omega_0) \times H^4(\Omega_0)$ by Lemma \ref{ds:lemma:regularity BHSTO}. Therefore, by definition the shape derivative $(\Psi'(\Vd), \omega'(\Vd)) \in H^4(\Omega_0) \times H^2(\Omega_0)$ exists.

 Now, let $\Vd \in \V^4$. Then from \cite[Proposition 3.1]{sokolowski1992introduction} we conclude
 \begin{align}
  \Delta \Psi'(\Vd) = -\omega'(\Vd)
  \qquad \mbox{in $\Omega_0$}
 \end{align}
 and
 \begin{align}
  \Delta \omega'(\Vd) = 0
  \qquad \mbox{in $\Omega_0$}.
 \end{align}
 For $\theta \in \Theta^4$ we have $\Psi(\theta) \circ (\I + \theta) = g$ on $\Gamma_0$ and thus by definition of the material derivative
 \begin{align}
  \dot{\Psi}(\Vd)|_{\Gamma_0} = 0.
 \end{align}
 Then,
 \begin{align}
 \begin{aligned}
  \Psi'(\Vd)|_{\Gamma_0} &= \dot{\Psi}(\Vd)|_{\Gamma_0} - (\Grad \Psi(0) \cdot \Vd)|_{\Gamma_0}
  \\
  &= - (\partial_\nv \Psi(0) (\nv \cdot \Vd))|_{\Gamma_0}
  \\
  &= 0,
  \end{aligned}
 \end{align}
 because $\Vd \in \V^4$ is normal and $\partial_\nv \Psi(0) = 0$ on $\Gamma_0$. Finally, we deduce from \cite[(3.12)]{sokolowski1992introduction}
 \begin{align}
  \partial_\nv \Psi'(\Vd) = \partial_\tv ((\nv \cdot \Vd) \partial_\tv \Psi(0)) + (\nv \cdot \Vd) \omega(0).
 \end{align}
 Then, $\partial_\tv ((\nv \cdot \Vd) \partial_\tv \Psi(0))$ vanishes because $\Vd = 0$ on $\Gamma_0^{in}$ and $\partial_\tv \Psi(0) = 0$ on $\Gamma_0^w$. We get
 \begin{align}
 \begin{aligned}
 \partial_\nv \Psi'(\Vd) &= 0
  \qquad &&\mbox{on $\Gamma_0^{in}$}
  \\
  \partial_\nv \Psi'(\Vd) &= (\nv \cdot \Vd) \omega(0)
  \qquad &&\mbox{on $\Gamma_0^w$}.
 \end{aligned}
 \end{align}
\end{proof}

\begin{lemma}
\label{lemma:ds stokes bd shape existence} 
 Let $z(\theta) = \omega(\theta)|_{\Gamma_\theta}$ for $\theta \in \Theta^4$. Then, the shape derivative $z'(\Vd) \in H^\frac{1}{2}(\Gamma_0)$ exists for $\Vd \in C^4(\R^2, \R^2)$ and is given by
 \begin{align}
  z'(\Vd) = \omega'(\Vd)|_{\Gamma_0} + \partial_\nv \omega(0) (\nv \cdot \Vd).
 \end{align}
\end{lemma}
\begin{proof}
 This is a direct consequence of Lemma \ref{shape:lemma:boundary shape derivative}.
\end{proof}

\begin{lemma}
The linearized shape operator $d\sostokes$ is well-defined and given by
\begin{align}
\begin{aligned}
 d\sostokes: \V^4 &\rightarrow L^2(\Gamma_0^w)
 \\
 \Vd &\mapsto \omega'(\Vd)|_{\Gamma_0^w} + \partial_\nv \omega(0) (\nv \cdot \Vd).
 \end{aligned}
\end{align}
\end{lemma}
\begin{proof}
 We have shown in Lemma \ref{lemma:ds stokes bd mat existence} that the operator is well-defined. Let $\Vd \in \V^4$. Remember that by definition $\Vd$ is normal on $\Gamma_0$. We conclude using Definition \ref{def:elliptic bd shape derivative} and Lemma \ref{lemma:ds stokes bd shape existence} 
\begin{align}
 d \sostokes(\Vd) = \dot{z}(\Vd) = z'(\Vd) + \partial_\tv z(0) (\tv \cdot \Vd)
 = \omega'(\Vd)|_{\Gamma^w_0} + \partial_\nv \omega(0) (\nv \cdot \Vd).
\end{align}
\end{proof}

\subsection{Approximate Controllability}
The approximate controllability of the operator $d\sostokes$ depends on the uniqueness question addressed in the following lemma. However, we can only show that the corresponding bilinear form is coercive but not that it is elliptic. Therefore, we have to rely on the weaker argument of Theorem \ref{pre:thm:Fredholm Alternative}, which states that the homogeneous solutions form a finite dimensional subspace. In the case that zero is no eigenvalue of the corresponding representation operator, this subspace is trivial. There is no way to tell whether zero is an eigenvalue of not. We know that there are only countably many eigenvalues which do not accumulate in a finite region (see \cite{wloka1987partial}).
\begin{lemma}
\label{zua:lemma:Stokes Uniqueness}
 Assume that $\Omega_0$ is bounded and of class $C^{4,1}$ and $c_{11} \in C^1(\Omegab_0)$. We consider
 \begin{align}
 \label{zua:eq:BiharmonicUniqueness}
 \begin{aligned}
  \Delta \Delta \phi &= 0
  \qquad &&\mbox{in $\Omega_0$}
  \\
  \phi &= 0
  \qquad &&\mbox{on $\Gamma_0$}
  \\
  \partial_\nv \phi &= 0
  \qquad &&\mbox{on $\Gamma_0^{in}$}
  \\
  \Delta \phi + c_{11} \partial_\nv \phi &= 0
  \qquad &&\mbox{on $\Gamma_0^w$}
 \end{aligned}
 \end{align}
 and define $\Z := \{ \phi \in H^4(\Omega_0); \phi \mbox{ solves \eqref{zua:eq:BiharmonicUniqueness}} \}$. Then, $\mathcal{Z}$ is a finite dimensional subspace of $H^4(\Omega_0)$.
\end{lemma}
\begin{proof}
Define $V := \{ u \in H^2(\Omega_0); u|_{\Gamma_0} = 0; \partial_\nv u|_{\Gamma_0^{in}} = 0 \}$. Let $\phi \in \Z$ and let $\eta \in V$ be a test function. Then,
 \begin{align}
 \label{zua:eq:Stokes Uniqueness integration}
  \begin{aligned}
   0 &= \int_{\Omega_0} \Delta \Delta \phi \eta \dint x
   \\
   &= \int_{\Omega_0} \Delta \phi \Delta \eta \dint x 
   + \int_{\Gamma_0} \partial_\nv \Delta \phi \eta \dint s
   - \int_{\Gamma_0} \Delta \phi \partial_\nv \eta \dint s
   \\
   &= \int_{\Omega_0} \Delta \phi \Delta \eta \dint x
   + \int_{\Gamma_0^w} c_{11} \partial_\nv \phi \partial_\nv \eta \dint s.
  \end{aligned}
 \end{align}
 We define the bilinear form
 \begin{align}
  a(\varphi, \eta) := \int_{\Omega_0} \Delta \varphi \Delta \eta \dint x
 \end{align}
 and the boundary form
 \begin{align}
  c(\varphi, \eta) := \int_{\Gamma_0^w} c_{11} \partial_\nv \varphi \partial_\nv \eta \dint s.
 \end{align}
The space $V$ is a closed subspace of $H^2(\Omega_0)$ with $H^2_0(\Omega_0) \subset V \subset H^2(\Omega)$ and $a(\varphi, \eta)$ is $V$-coercive (cf. Definition \ref{def:V-Coercive} and \cite{wloka1987partial}). Because of $c_{11} \in C^1(\Omegab_0)$ the bilinear form $a(\varphi, \eta) + c(\varphi, \eta)$ is also $V$-coercive (see \cite{wloka1987partial}). The embedding $V \embed L^2(\Omega_0) \embed V'$ is a Gelfand triple and $V \embed L^2(\Omega_0)$ is compact (see \cite{wloka1987partial}). Thus, the assumptions of Theorem \ref{pre:thm:Fredholm Alternative} hold for the weak formulation: 

Find $\varphi \in V$ such that
 \begin{align}
 \label{zua:eq:Stokes Uniqueness weak}
  a(\varphi, \eta) + c(\varphi, \eta) = 0
  \qquad \mbox{for all $\eta \in V$}.
 \end{align}
 From Theorem \ref{pre:thm:Fredholm Alternative} we conclude that $\tilde{\Z} := \{ \varphi \in V; \varphi \mbox{ solves \eqref{zua:eq:Stokes Uniqueness weak}}\}$ is finite dimensional. Because of \eqref{zua:eq:Stokes Uniqueness integration} we know that every $\phi \in \Z$ solves \eqref{zua:eq:Stokes Uniqueness weak} and thus we conclude
 \begin{align}
  \Z = \tilde{\Z} \cap H^4(\Omega_0)
 \end{align}
 which yields the result.
\end{proof}

The next lemma shows the regularity of the coefficient appearing in the approximate controllability proof.
\begin{lemma}
\label{ds:lemma:regularity c11}
 Assume that $\omega(0) \neq 0$ on $\Gamma_0^w$. Then,
 \begin{align}
  c_{11} := - \frac{\partial_\nv \omega(0)}{\omega(0)} \in C^1(\Gamma_0^w).
 \end{align}
\end{lemma}
\begin{proof}
 We have shown that $\Psi(0) \in H^6(\Omega_0)$ and thus $\omega(0) \in H^4(\Omega_0)$. Then, $\omega|_{\Gamma_0} \in H^{5+\frac{1}{2}}(\Gamma_0)$ and $\partial_\nv \omega|_{\Gamma_0} \in H^\frac{5}{2}(\Gamma_0)$. By the Lemma of Sobolev (see \cite{wloka1987partial}) we have  $\omega|_{\Gamma_0} \in C^1(\Gamma_0)$ and $\partial_\nv \omega|_{\Gamma_0} \in C^1(\Gamma_0)$ and since $\omega$ is non-zero on $\Gamma_0^w$, $c_{11} \in C^1(\Gamma_0^w)$ holds.
\end{proof}

Finally, we are prepared to show the main result for the operator $d\sostokes$.
\begin{proof}[of Theorem \ref{ds:thm:approximate control stokes}]
Define
\begin{align}
 H^\frac{5}{2}_{i=0}(\Gamma_0) = \{\mu \in H^\frac{5}{2}(\Gamma_0); \mu = 0 \mbox{ on $\Gamma_0^{in}$} \}
\end{align}
 and for $\mu \in H^\frac{5}{2}_{i=0}(\Gamma_0)$ let $\phi(\mu) \in H^4(\Omega_0)$ be the solution of the adjoint problem
 \begin{align}
 \begin{aligned}
  \Delta \Delta \phi(\mu) &= 0
  \qquad &&\mbox{in $\Omega_0$}
  \\
  \phi(\mu) &= 0
  \qquad &&\mbox{on $\Gamma_0$}
  \\
  \partial_\nv \phi(\mu) &= \mu
  \qquad &&\mbox{on $\Gamma_0$}
 \end{aligned}
 \end{align}
 where the existence and regularity follows from \cite{wloka1987partial}. For $(\Vd, \mu) \in \V^4 \times H^\frac{5}{2}_{i=0}(\Gamma_0)$ integration by parts yields
 \begin{align}
  \begin{aligned}
   0 &= \int_{\Omega_0} \Delta \Delta \Psi'(\Vd) \phi(\mu) \dint x
   \\
   &= \int_{\Omega_0} \Delta \Psi'(\Vd) \Delta \phi(\mu) \dint x
   - \int_{\Gamma_0} \Delta \Psi'(\Vd) \partial_\nv \phi(\mu) \dint s
   \\
   &= \int_{\Omega_0} \Psi'(\Vd) \Delta \Delta \phi(\mu) \dint x
   - \int_{\Gamma_0} \Delta \Psi'(\Vd) \partial_\nv \phi(\mu) \dint s
   + \int_{\Gamma_0} \partial_\nv \Psi'(\Vd) \Delta \phi(\mu) \dint s
  \end{aligned}
 \end{align}
 and we get the identity
 \begin{align}
  -\int_{\Gamma_0^w} \omega'(\Vd) \mu \dint s
  = \int_{\Gamma_0^w} (\nv \cdot \Vd) \omega(0) \Delta \phi(\mu) \dint s.
 \end{align}

Now, assume that $\mu \in \im(d\sostokes)^\perp \cap H^\frac{5}{2}_{i=0}(\Gamma_0)$, i.e.,
\begin{align}
 \int_{\Gamma_0^w} d\sostokes(\Vd) \mu \dint s = 0
 \qquad \mbox{for all $\Vd \in \V^4$}.
\end{align}
We conclude
\begin{align}
 \begin{aligned}
  0 &= \int_{\Gamma_0^w} d\sostokes(\Vd) \mu \dint s
  \\
  &= \int_{\Gamma_0^w} \omega'(\Vd) \mu \dint s
  + \int_{\Gamma_0^w} \partial_\nv \omega(0) (\nv \cdot \Vd) \mu \dint s
  \\
  &= \int_{\Gamma_0^w} (\nv \cdot \Vd)(-\omega(0) \Delta \phi(\mu) 
  + \partial_\nv \omega(0) \partial_\nv \phi(\mu)) \dint s.
 \end{aligned}
\end{align}
Since $\{\nv \cdot \Vd; \Vd \in \V^4\}$ is dense in $L^2(\Gamma_0^w)$ we derive 
\begin{align}
 -\omega(0) \Delta \phi(\mu) 
  + \partial_\nv \omega(0) \partial_\nv \phi(\mu) = 0
 \qquad \mbox{on $\Gamma_0^w$}.
 \end{align}
 Because of $\omega(0) = \sostokes(0) \neq 0$ on $\Gamma_0^w$, we can define
 \begin{align}
  c_{11} := -\frac{\partial_\nv \omega(0)}{\omega(0)} \in C^1(\Gamma_0^w),
 \end{align}
 where the regularity follows from Lemma \ref{ds:lemma:regularity c11}. This yields the uniqueness problem
\begin{align}
\label{ds:eq:approximate control stokes eq1}
 \begin{aligned}
  \Delta \Delta \phi(\mu) &= 0
  \qquad &&\mbox{in $\Omega_0$}
  \\
  \phi(\mu) &= 0
  \qquad &&\mbox{on $\Gamma_0$}
  \\
  \partial_\nv \phi(\mu) &= 0
  \qquad &&\mbox{on $\Gamma_0^{in}$}
  \\
  \Delta \phi(\mu) + c_{11}\phi(\mu) &= 0
  \qquad &&\mbox{on $\Gamma_0^w$}.
 \end{aligned}
 \end{align}
 Define
 \begin{align}
  \mathcal{Z} := \{\phi(\mu) \in H^4(\Omega_0); \mbox{$\phi(\mu)$ is solution of \eqref{ds:eq:approximate control stokes eq1}} \}
 \end{align}
 and
 \begin{align}
  \mathcal{Z}_{\partial_\nv} := \{ \mu = \partial_\nv \phi|_{\Gamma_0^w}; \phi \in \mathcal{Z} \}.
 \end{align}
 By Lemma \ref{zua:lemma:Stokes Uniqueness} we know that $\mathcal{Z}$ is a finite dimensional subspace of $H^4(\Omega_0)$. Then, $\mathcal{Z}_{\partial_\nv}$ is a finite dimensional subspace of $H^\frac{5}{2}_{i=0}(\Gamma_0)|_{\Gamma_0^w}$ and thus of $L^2(\Gamma_0^w)$. Using Lemma \ref{lemma:approximately controllable} we conclude that $d\sostokes$ is approximately controllable as a mapping to $L^2(\Gamma_0^w)/\mathcal{Z}_{\partial_\nv}$.
\end{proof}

\section{Conclusion}
\label{sec:Conclusion}
We have studied the controllability of two shape-dependent operators based on flow problems. We were able to prove approximate controllability for linearizations of these operators using an adjoint argument. For the Stokes operator we have to note that a small subspace remains which is not controllable, but this subspace is finite dimensional. Even though we have studied linearizations, we can draw conclusions for the actual operators. Having the approximate controllability property for the linearization means that we can change the observable into almost every direction by applying infinitesimal shape perturbations. Our application in view is the design of polymer distributors with specific wall shear stress profiles. Theorem \ref{ds:thm:approximate control stokes} does suggest that the space of reachable wall shear stress profiles is rather large. Therefore, we can expect a good performance of the shape optimization algorithm, meaning that the optimal stress profiles lie close to the desired target stress in the \mbox{$L^\infty$-sense}, even though we are only using $L^2$ shape optimization. This statement does agree with our numerical experience form \cite{leithauser2013controllability}, where we have solved an optimization problem based on the Stokes operator.

\begin{appendix} 
\section{Shape Differentiation}
\label{sec:appendix:shape diff}
We provide the concepts of material and shape derivatives and cite the essential theory on the differentiation of shape-dependent integrals. Further details can be found in \cite{sokolowski1992introduction}.
\begin{definition}[Material Derivative]
\label{def:material derivative}
  Let $y(\theta) \in H^m(\Omega_\theta)$ be given for $\theta \in \Theta^k$. Then, $\dot{y}(\Vd)$ is called material derivative of $y(\theta)$ in direction $\Vd \in C^{k,1}(\R^2, \R^2)$ if and only if the limit
 \begin{align}
  \dot{y}(\Vd) = \lim_{s\rightarrow 0} \frac{1}{s} \big( y(s\Vd) \circ (\I+s\Vd) - y(0) \big)  \in H^m(\Omega_0)
 \end{align}
 exists.
\end{definition}

The material derivative of a boundary function is defined in a similar way:
\begin{definition}[Boundary Material Derivative]
\label{def:material derivative bd}
 Let $z(\theta) \in H^r(\Gamma_\theta)$ be given for $\theta \in \Theta^k$. Then, $\dot{z}(\Vd)$ is called material derivative in direction of $\Vd \in C^{k,1}(\R^2, \R^2)$ if and only if the limit
  \begin{align}
  \dot{z}(\Vd) = \lim_{s\rightarrow 0} \frac{1}{s} \big( z(s\Vd) \circ (\I + s\Vd) - z(0) \big) \in H^r(\Gamma_0)
 \end{align}
 exists.
\end{definition}

The following relation holds between the material derivatives and the boundary material derivative:
\begin{lemma}[from \cite{sokolowski1992introduction}]
\label{shape:lemma:boundary material derivative}
 Let $k \geq m \geq 1$. Let $y(\theta) \in H^m(\Omega_\theta)$ and let $z(\theta) = y(\theta)|_{\Gamma_\theta} \in H^{m-\frac{1}{2}}(\Gamma_\theta)$ for $\theta \in \Theta^k$. Suppose that the material derivative $\dot{y}(\Vd) \in H^m(\Omega_0)$ exists for $\Vd \in C^{k,1}(\R^2, \R^2)$. Then, the material derivative of the boundary function exists and is given by $\dot{z}(\Vd) = \dot{y}(\Vd)|_{\Gamma_0} \in H^{m-\frac{1}{2}}(\Gamma_0)$.
\end{lemma}

Next, we define the shape derivative. The difference between material and shape derivative is that the first is the derivative of $y(\theta) \circ (\I+\theta)$ and the second the derivative of just $y(\theta)$ without the pull-back. It is convenient to derive the definition of the shape derivative from the material derivative by just subtracting the part originating from differentiating the map $(\I+\theta)$. This way,  we can directly derive the existence from the existence of the material derivative.
\begin{definition}[Shape Derivative]
\label{def:elliptic shape derivative}
 Let $y(\theta) \in H^m(\Omega_\theta)$ for $\theta \in \Theta^k$. Assume that the material derivative $\dot{y}(\Vd) \in H^m(\Omega_0)$ exists for $\Vd \in C^{k,1}(\R^2, \R^2)$. Then, shape derivative in direction $\Vd$ is defined by
 \begin{align}
  y'(\Vd) := \dot{y}(\Vd) - \Grad y(0) \cdot \Vd \in H^{m-1}(\Omega_0).
 \end{align}
 Furthermore, we can see directly from the definition that $y(0) \in H^{m+1}(\Omega_0)$ implies $y'(\Vd) \in H^m(\Omega_0)$.
\end{definition}

On the boundary we define the shape derivative in the following way:
\begin{definition}[Boundary Shape Derivative]
\label{def:elliptic bd shape derivative}
 Let $z(\theta) \in H^r(\Gamma_\theta)$ for $\theta \in \Theta^k$ and assume that the material derivative $\dot{z}(\Vd) \in H^r(\Gamma_0)$ exists for $\Vd \in C^{k,1}(\R^2, \R^2)$. Then, the shape derivative in direction $\Vd$ is defined by
 \begin{align}
  z'(\Vd) := \dot{z}(\Vd) - \partial_\tv z(0) \, \tv \cdot \Vd \in H^{r-1}(\Gamma_0).
 \end{align}
 Furthermore, if $z(0) \in H^{r+1}(\Gamma_0)$, then $z'(\Vd) \in H^r(\Gamma_0)$.
\end{definition}

The following lemma draws a connection between shape derivatives on the domain and the boundary:
\begin{lemma}[from \cite{sokolowski1992introduction}]
\label{shape:lemma:boundary shape derivative}
 Let $k \geq m \geq 1$. For $\theta \in \Theta^k$ let $y(\theta) \in H^{m}(\Omega_\theta)$ and $z(\theta) = y(\theta)|_{\Gamma_\theta} \in H^{m-\frac{1}{2}}(\Gamma_\theta)$ and suppose that $y(0) \in H^{m+1}(\Omega_0)$. Assume that $y'(\Vd) \in H^{m}(\Omega_0)$ exists for $\Vd \in C^{k,1}(\R^2, \R^2)$. Then,
 \begin{align}
  z'(\Vd) = y'(\Vd)|_{\Gamma_0} + \partial_\nv y(0) (\Vd \cdot \nv) \in H^{m-\frac{1}{2}}(\Gamma_0).
 \end{align}
\end{lemma}

For the pull-back of the Laplacian the following holds:
\begin{lemma}[from \cite{simon1980differentiation}]
\label{ds:lemma:gradf_TV}
 For $k \geq m \geq 2$ let $\theta \in \Theta^k$. Then
 \begin{align}
  (\Delta f) \circ (\I + \theta) = \Delta_\theta (f \circ (\I + \theta))
 \end{align}
 for all $f \in H^m(\Omega_\theta)$, where $\Delta_\theta: H^m(\Omega_0) \rightarrow H^{m-2}(\Omega_0)$ is defined by
 \begin{align}
  \Delta_\theta f := \sum_{i, j, l=1}^d M_{ij}(\theta) \frac{\partial}{\partial x_j}
  \left( M_{il}(\theta) \frac{\partial f}{\partial x_l} \right)
 \end{align}
 with $M(\theta) := \left[(D(\I+\theta))^{-1}\right]^T$.
\end{lemma}

The following results provide the derivatives of integral expressions:
\begin{lemma}[Differentiation of Domain Integrals, see \cite{sokolowski1992introduction}]
\label{zua:lemma:diff domain int}
Let $k \geq 1$. For $\theta \in \Theta^k$ let $y(\theta) \in H^1(\Omega_\theta)$ and $f \in H^1(\R^2)$, let the shape derivative $y'(\Vd) \in H^1(\Omega_0)$ exist for $\Vd \in C^k(\R^2, \R^2)$ and let
\begin{align}
 J(\theta) := \int_{\Omega_\theta} y(\theta) f \dint x.
\end{align}
Then, the derivative of $J(\theta)$ with respect to $\theta$ in direction $\Vd \in C^k(\R^2, \R^2)$ is given by
\begin{align}
 dJ(\Vd) := \frac{d J(\theta)}{d \theta}(0) \Vd
 = \int_{\Omega_0} y'(\Vd) f \dint x + \int_{\Gamma_0} y(0) f (\Vd\cdot\nv) \dint s.
\end{align}
\end{lemma}

\begin{lemma}[Differentiation of Boundary Integrals, see \cite{sokolowski1992introduction}]
\label{zua:lemma:diff boundary int}
 Let $k \geq 2$. For $\theta \in \Theta^k$ let $z(\theta) \in H^\frac{3}{2}(\Gamma_\theta)$ be shape differentiable with derivative $z'(\Vd) \in H^\frac{3}{2}(\Gamma_0)$ for $\Vd \in C^k(\R^2, \R^2)$ and let $f \in H^2(\R^2)$. Define
 \begin{align}
  J(\theta) = \int_{\Gamma_\theta} z(\theta) f \dint s.
 \end{align}
 Then, the derivative of $J(\theta)$ with respect to $\theta$ in direction $\Vd \in C^k(\R^2, \R^2)$ is given by
 \begin{align}
  dJ(\Vd) = \int_{\Gamma_0} z'(\Vd) f + (z(0) \partial_\nv f + \kappa z(0) f) (\Vd\cdot\nv) \dint s.
 \end{align}
 In particular if $z(\theta) = y(\theta)|_{\Gamma_\theta}$ for $y(\theta) \in H^2(\Omega_\theta)$ with $y'(\Vd) \in H^2(\Omega_0)$ we have by Lemma \ref{shape:lemma:boundary shape derivative}
 \begin{align}
  dJ(\Vd) = \int_{\Gamma_0} y'(\Vd) f + (\partial_\nv y(0) f + z(0) \partial_\nv f + \kappa z(0) f) (\Vd\cdot\nv) \dint s.
 \end{align}
\end{lemma}

The existence proofs for the material derivatives rely on the implicit function theorem:
\begin{theorem}[Implicit Function Theorem, from \cite{amann2008analysis2}]
\label{thm:implicit function}
 Let $E_1$, $E_2$, $F$ be Banach spaces, let $W$ be open in $E_1 \times E_2$ and let $f \in C^q(W, F)$. Suppose that $(x_0, y_0) \in W$ such that $f(x_0, y_0) = 0$ and
 \begin{align}
  D_2 f(x_0, y_0): E_2 \rightarrow F
 \end{align}
 is an isomorphism. Then, there are open neighborhoods $U \subset W$ of $(x_0, y_0)$ and $V \subset E_1$ of $x_0$ and a unique $\mathcal{G} \in C^q(V, E_2)$ such that
 \begin{align}
  \mbox{($(x, y) \in U$ and $f(x, y) = 0$)}
  \Leftrightarrow
  \mbox{($x \in V$ and $y = \mathcal{G}(x)$)}.
 \end{align}
\end{theorem}

\section{Existence and Uniqueness of Solutions for PDE}
\label{sec:appendix:existence}
Based on the notation of \cite{wloka1987partial} we introduce elliptic and coercive bilinear forms which give rise to usual existence existence results for partial differential equations.

\begin{definition}[$V$-Elliptic]
\label{def:V-Elliptic}
 Let $m \geq 1$ and let $V$ be a closed subspace equipped with the $H^m(\Omega)$-norm between $H^m_0(\Omega) \subset V \subset H^m(\Omega)$. We call a bilinear form $a: H^m(\Omega) \times H^m(\Omega) \rightarrow \R$ $V$-elliptic if and only if
 \begin{enumerate}
  \item $\abs{a(\psi, \phi)} \leq c_1 \vectornorm{\psi}_{H^m(\Omega)} \vectornorm{\phi}_{H^m(\Omega)}$, for all $\psi, \phi \in H^m(\Omega)$
  \item $a(\psi, \psi) \geq c_2 \vectornorm{\psi}^2_{H^m(\Omega)}$, for all $\psi \in V$
 \end{enumerate}
 where $c_1, c_2 > 0$  are independent of $\psi$ and $\phi$.
\end{definition}

\begin{definition}[$V$-Coercive]
\label{def:V-Coercive}
 Let $m \geq 1$ and let $V$ be a closed subspace equipped with the $H^m(\Omega)$-norm between $H^m_0(\Omega) \subset V \subset H^m(\Omega)$. We call a bilinear form $a: H^m(\Omega) \times H^m(\Omega) \rightarrow \R$ $V$-coercive if and only if
 \begin{enumerate}
  \item $\abs{a(\psi, \phi)} \leq c_1 \vectornorm{\psi}_{H^m(\Omega)} \vectornorm{\phi}_{H^m(\Omega)}$, for all $\psi, \phi \in H^m(\Omega)$
  \item $a(\psi, \psi) + k \vectornorm{\psi}_{L^2(\Omega)}^2 \geq c_2 \vectornorm{\psi}^2_{H^m(\Omega)}$, for all $\psi \in V$
 \end{enumerate}
 where $c_1, c_2 > 0$ and $k \in \R$ are constants independent of $\psi$ and $\phi$.
\end{definition}

For $V$-elliptic problems we can apply the Lax-Milgram Lemma to provide the existence of a unique solution.
\begin{theorem}[Lax-Milgram, from \cite{wloka1987partial}]
\label{elliptic:thm:existence}
 Let $a(\psi, \phi)$ be $V$-elliptic and let $f \in V'$. Then there exists a unique $\psi \in V$ such that
 \begin{align}
  \label{elliptic:eq:existence}
  a(\psi, \phi) = \dual{f, \phi}_{V', V}
 \end{align}
 for all $\phi \in V$.
\end{theorem}

A key part in our line of proof is the uniqueness question addressed in Lemma \ref{zua:lemma:Stokes Uniqueness}. However, since the corresponding bilinear form is only $V$-coercive we rely on the following theorem, which does not provide uniqueness, but states that the space of homogeneous solutions is finite dimensional.
\begin{theorem}[from \cite{wloka1987partial}]
\label{pre:thm:Fredholm Alternative}
 Let $V \embed L^2(\Omega) \embed V'$ be a Gelfand triple and let the embedding $V \embed L^2(\Omega)$ be compact. Let $a(\psi, \phi)$ be $V$-coercive, then
 \begin{align}
   \mathcal{Z} = \{\psi \in V; a(\psi, \phi) = 0 \mbox{ for all $\phi \in V$}\}
  \end{align}
 is a finite dimensional subspace of $V$. Furthermore, if $0$ is no eigenvalue of the corresponding representation operator then $\mathcal{Z} = \{0\}$ holds.
\end{theorem}

\end{appendix}

  \section*{Acknowledgments}
This work was supported by the German Federal Ministry of Education and Research (BMBF) grant no. 03MS606F.


\end{document}